\documentclass[11pt]{amsart}

\usepackage{times,amsfonts,amsmath,amstext,amsbsy,amssymb,
  amsopn,amsthm,upref,eucal, amscd, color}
\usepackage[T1]{fontenc}

\newtheorem{theorem}{Theorem}[section]

\newtheorem{proposition}[theorem]{Proposition}

\numberwithin{equation}{section}

\theoremstyle{definition}
\newtheorem{definition}[theorem]{Definition}

\newcommand{\field}[1]{\mathbb{#1}}

\newcommand{\C}{\field{C}}

\begin{document}
\title[Explicit error term in Theorem A]{An explicit error term in Theorem A}
\author{Nicolas Bergeron}
\address{Nicolas Bergeron: Universit\'e Pierre et Marie Curie, Institut de Math\'ematiques de Jussieu, CNRS (UMR 7586), 4, place Jussieu 75252 Paris Cedex 05, France} 
\email{nicolas.bergeron@imj-prg.fr}
\author{Carlos Matheus}
\address{Carlos Matheus: CMLS, \'Ecole Polytechnique, CNRS  (UMR  7640),  91128  Palaiseau,  France}
\email{carlos.matheus@math.cnrs.fr}
\date{\today}

\maketitle

\section{Introduction}\label{s.intro}



Recall that Theorem A above ensures the existence of a constant $\delta>0$ such that the number $N(V)$ of sLag fibrations with volume $\leq V$ in a generic twistor family of K3 surfaces is 
\begin{equation}\label{e.Filip}
N(V) = C\cdot V^{20} + O(V^{20-\delta})
\end{equation} 
where $C>0$ is the ratio of volumes of two concrete homogenous spaces. 

The goal of this appendix is to prove that $\delta$ can be taken to be $\left( \frac{4}{697633} \right)^-$:

\begin{theorem}\label{t.A} In the same setting as Theorem A above, one actually has 
$$N(V) = C\cdot V^{20}+O_{\varepsilon}(V^{\frac{13952656}{697633}+\varepsilon})$$
for all $\varepsilon>0$. 
\end{theorem}

\section{Reduction of Theorem \ref{t.A} to dynamics in homogenous spaces}\label{s.reduction}

Filip derived his counting formula \eqref{e.Filip} from certain equidistribution results. More precisely, let  $\Lambda_{\mathbb{Z}}$ be a lattice isomorphic to $H^2(S,\mathbb{Z})$, where $S$ is a K3 surface. Fix $P\subset\Lambda_{\mathbb{R}}$ a positive-definite 3-plane. 
Denote by $\Lambda_{\mathbb{Z}}^0$ the set of primitive isotropic integral vectors and fix $e\in\Lambda_{\mathbb{Z}}^0$. For each $v\in\Lambda_{\mathbb{R}}=P\oplus P^{\perp}$, let $v:=(v)_P\oplus(v)_{P^{\perp}}$ with $(v)_P\in P$ and $(v)_{P^{\perp}}\in P^{\perp}$. Consider the orthogonal group  $G:=O(\Lambda_{\mathbb{R}})$, the lattice $\Gamma:=O(\Lambda_{\mathbb{Z}})$ and the maximal compact subgroup $K:=O(P)\times O(P^{\perp})$ of $G$, and, for a fixed $e\in\Lambda_{\mathbb{Z}}^0$, denote by $H_e:=Stab_G(e)$ and $\Gamma_e=Stab_{\Gamma}(e)$. 

The volumes of the locally homogenous spaces $X:=\Gamma\backslash G$ and $Y:=\Gamma_e\backslash H_e$ are finite. As it is observed in \cite[pp. 4]{Filip}, the constants $C>0$ and $\delta>0$ in \eqref{e.Filip} are the constant described in \cite[Theorem 3.1.3]{Filip}. In particular, 
$$C=\frac{\textrm{Vol }Y}{20\cdot\textrm{Vol }X}$$ 

The constant $\delta>0$ is related to the dynamics of a certain one-parameter subgroup $a_t$ of $G\simeq  SO(3,19)(\mathbb{R})$. More concretely, given $e$ and $P$ as above, let $e'$ be the isotropic vector given by 
$$e':=(e)_{P}\oplus -(e)_{P^{\perp}} \quad \textrm{where} \quad e:=(e)_P\oplus (e)_{P^{\perp}}$$ 
In this context, we denote by $\{a_t\}_{t\in\mathbb{R}}\subset G$ the one-parameter subgroup defined as 
$$a_t\cdot e = \exp(-t)\cdot e, \quad a_t\cdot e' = \exp(t)\cdot e', \quad a_t|_{(e\oplus e')^{\perp}} = \textrm{id}.$$
It is explained in \cite[Subsection \textcolor{red}{3.6.9}]{Filip} that\footnote{Indeed, \cite[pp. 29]{Filip} says that the optimal choice of $\delta$ occurs precisely when the terms $\varepsilon e^{20T} = e^{(20-\delta)T}$ and $\varepsilon^{-d_{l_0}} e^{(20-\delta_0)T} = e^{(20-\delta_0+\delta d_{l_0})T}$ have the same order in $T$.} the quantity $\delta$ in \eqref{e.Filip} is  
\begin{equation}\label{e.delta}
\delta = \frac{\delta_0}{d_{l_0}+1}
\end{equation}
where $(d_l)_{l\in\mathbb{N}}$ are the exponents in \cite[Proposition 3.5.10 (ii)]{Filip}, and $\delta_0>0$, $l_0\in\mathbb{N}$ are the constants in the following equidistribution statement in \cite[Theorem 4.3.1]{Filip}:  
\begin{equation}\label{e.thm431}
\int_{Y a_t} w \,\, d\mu_{Y a_t} = \frac{\textrm{Vol }Y}{\textrm{Vol }X}\int_X w \,\, d\mu_X + O(\|w\|_l \,\, e^{-\delta_0 t})
\end{equation} 
for all Sobolev scales $l\geq l_0$ (see \cite[\S 4.2.2]{Filip} for the definition of the Sobolev norms in this context). 

A quick inspection of the proof of \cite[Proposition 3.5.10 (ii)]{Filip} (related to the thickening of $K$) reveals that the exponents $d_l$ depend linearly on $l$. In fact, the constant $c_1(l)$ in \cite[Equation (3.5.15)]{Filip} gives the power of $\varepsilon$ associated to the volumes of $\varepsilon$-balls at the origin of $\mathfrak{p}_{\mathfrak{m}}\times\mathfrak{n}^+\times\mathfrak{a}$, that is, $c_1(l) = \textrm{dim}(G)-\textrm{dim}(K)$ (and, hence, $c_1(l)$ independs of $l$). Since the $l$-th derivative of $\chi_{\varepsilon}$ is bounded by a multiple of $\varepsilon^{-c_1(l)-l}$ and it is supported in a $\varepsilon$-neighborhood of $K$, the $l$-Sobolev norm of $\chi_{\varepsilon}$ is bounded by a multiple of $\varepsilon^{-l-c_1(l)/2}$. Therefore, 
\begin{equation}\label{e.dl} 
d_l:=l+\frac{\textrm{dim}(G)-\textrm{dim}(K)}{2}.
\end{equation}

\section{Equidistribution and rates of mixing}\label{s.mixing}

The constants $\delta_0>0$ and $l_0\in\mathbb{N}$ in \eqref{e.thm431} are described in \cite[pp. 36]{Filip} and they are related to the geometry of $Y\subset X$ and the rate of mixing of $a_t$. 

\subsection{Injectivity radius} We denote by $\textrm{inj}(x)$ the local injectivity radius at a point $x \in X$ and we let $Y_{\varepsilon}:=\{y\in Y: \textrm{inj}(y)\geq \varepsilon\}$. By \cite[Proposition 4.1.3]{Filip}, we know that the arguments of \cite[Lemma 11.2]{BO} provide a constant $p>0$ such that $\mu_Y(Y\setminus Y_{\varepsilon})=O(\varepsilon^p)$. Actually, a close inspection of these arguments (of integration over Siegel sets) reveal that $p=1$ in our specific setting (of $G\simeq SO(3,19)(\mathbb{R})$):
\begin{equation}\label{e.p}
\mu_Y(Y\setminus Y_{\varepsilon})=O(\varepsilon)
\end{equation}

\subsection{Thickening of $Y$} Let us fix some parameter $0<p'<1$ (very close to one in practice) and consider \cite[Proposition 4.1.6]{Filip} (of thickening of $Y$) where it is constructed a family of smooth versions $\phi_{\varepsilon}$ of the characteristic function of $Y$. As it turns out, $\phi_{\varepsilon}$ is the product of two functions: $\tau_{\varepsilon}$ is a bump function supported\footnote{In fact, Filip sets $p'=1/2$ for his construction of $\tau_{\varepsilon}$, but any value of $0<p'<1$ can be taken here: indeed, the construction of $\tau_{\varepsilon}$ can be made as soon as the  local product structure statement \cite[Proposition 4.1.5]{Filip} holds (and this is the case for any choice of $0<p'<1$ because $\varepsilon^{p'}\gg 2\varepsilon$ for all sufficiently small $\varepsilon>0$).} on $Y_{\varepsilon^{p'}}$ and $\rho_{\varepsilon}$ is a bump function supported on the $\varepsilon$-neighborhood of the identity in a certain Lie group $N'$ of dimension $\textrm{dim}(N') = \textrm{dim}(X)-\textrm{dim}(Y)$. 

The bump function $\rho_{\varepsilon}$ is obtained by rescaling of a fixed smooth bump function on $N'$, so that its $l$-th Sobolev norm satisfies 
$\|\rho_{\varepsilon}\|_{l}=O(\varepsilon^{-l-\frac{\textrm{dim}(X)-\textrm{dim}(Y)}{2}})$. 

The function $\tau_{\varepsilon}$ is 
$$\tau_{\varepsilon} = \frac{\sum_{y_j\in\mathcal{F}}\beta_{y_j,\varepsilon}}{\sum_{y_i\in\mathcal{G}}\beta_{y_i,\varepsilon}} := \frac{\sum_{y_j\in\mathcal{F}}\beta_{y_j,\varepsilon}}{\beta_{\mathcal{G},\varepsilon}}$$
where $\{y_k\}\subset Y_{\varepsilon^{p'}}$ is a maximal collection of points such that the balls $B(y_k,\varepsilon^3)\subset Y$ are mutually disjoint, $\mathcal{F}=\{y_k\}\cap Y_{4\varepsilon^{p'}}$,  $\mathcal{G}=\{y_k\}\cap Y_{2\varepsilon^{p'}}$, and the functions $\beta_{.,\varepsilon}$ are translates of a bump function $\beta_{\varepsilon}$ whose $l$-th Sobolev norm is $\|\beta_{\varepsilon}\|_l=O(\varepsilon^{-l+\frac{\textrm{dim}(Y)}{2}})$. 

On one hand, since a ball $B$ of radius $\varepsilon$ at a point of $Y_{\varepsilon^{p'}}$ has volume $O(\varepsilon^{\textrm{dim}(Y)})$, the cardinality of $\mathcal{G}\cap B$ is $O(\varepsilon^{-2\textrm{ dim}(Y)})$, the arguments in \cite[pp. 1928]{BO} imply that the $L^{\infty}$-norm of the first $l$ derivatives of $1/\beta_{\mathcal{G},\varepsilon}$ is $O(\varepsilon^{-l-2\textrm{ dim}(Y)})$. On the other hand, the cardinality of $\mathcal{F}$ is $O(\varepsilon^{-3\textrm{ dim}(Y)})$ and $\|\beta_{y_j,\varepsilon}\|_l=\|\beta_{\varepsilon}\|_l$. It follows that 
$$\|\tau_{\varepsilon}\|_l = O(\varepsilon^{-l-\frac{9\textrm{ dim}(Y)}{2}}).$$

By inserting these facts into the definition of $\phi_{\varepsilon}$ in \cite[Equation (4.1.7)]{Filip}, we deduce from Sobolev's lemma that 
\begin{equation}\label{e.Cl}
\|\phi_{\varepsilon}\|_l=O(\varepsilon^{-2l-4\textrm{ dim}(Y)-\frac{\textrm{dim}(X)}{2}}),
\end{equation}
for all $l>\textrm{dim}(X)/2$, that is, the constant $C_l$ in \cite[Proposition 4.1.6 (iii)]{Filip} is 
$$C_l:=2l+ 4\textrm{ dim}(Y)+\frac{\textrm{dim}(X)}{2}$$ 

For later use, notice that $\phi_{\varepsilon}$ verifies $\int_X \phi_{\varepsilon} \,\, d\mu_X = \textrm{Vol } Y + O(\textrm{Vol}(Y\setminus Y_{\varepsilon^{p'}}))$. By combining this estimate with \eqref{e.p}, we get 
\begin{equation}\label{e.phi-mass}
\int_X \phi_{\varepsilon} \,\, d\mu_X = \textrm{Vol } Y + O(\varepsilon^{p'})
\end{equation}

\subsection{Wavefront lemma} The proof of Lemma 4.1.10 in \cite{Filip} says that 
$$\int_X w\cdot(\phi_{\varepsilon}\cdot a_t) d\mu_X = \int_Y w(ya_t) d\mu_Y(y) + O(\varepsilon \textrm{Lip}(w))+O(\varepsilon^{p p'}\|w\|_{L^{\infty}})$$ 
where $p>0$ is the parameter such that $\mu_Y(Y\setminus Y_{\varepsilon^{p'}})=O(\varepsilon^{p p'})$. Therefore, we deduce from \eqref{e.p} and Sobolev's lemma that 
\begin{equation}\label{e.wavefront} 
\int_X w\cdot(\phi_{\varepsilon}\cdot a_t) d\mu_X = \int_Y w(ya_t) d\mu_Y(y) + O(\varepsilon^{p'} \|w\|_l)
\end{equation} 
for all $l>1+\textrm{dim}(X)/2$. 

\subsection{Reduction of equidistribution to rate of mixing} By following \cite[pp. 36]{Filip}, let us compute the constants $\delta_0>0$ and $l_0\in\mathbb{N}$ in \eqref{e.thm431} in terms of the following quantitative mixing statement: there exists $\delta_0'>0$ such that 
\begin{equation}\label{e.mixing} 
\left|\int_X \alpha\cdot(\beta\cdot g) d\mu - \left(\int_X \alpha d\mu\right)\left(\int_X \beta d\mu\right)\right| = O(\|\alpha\|_l \|\beta\|_l \|g\|^{-\delta_0'})
\end{equation} 
for all $l\geq l_0'$. (Here, $\mu=\mu_X/\textrm{Vol }X$ is the normalized Haar measure.) 

For this sake, we observe that \eqref{e.mixing} says that 
$$\int_X w\cdot(\phi_{\varepsilon}\cdot a_t) d\mu_X = \frac{1}{\textrm{Vol }X}\left(\int_X w d\mu_X\right)\left(\int_X \phi_{\varepsilon} d\mu_X\right) + O(\|w\|_l \|\phi_{\varepsilon}\|_l e^{-t\delta_0'})$$ 
for $l\geq l_0'$. 

By \eqref{e.Cl} and \eqref{e.phi-mass}, the previous estimate implies 
$$\int_X w\cdot(\phi_{\varepsilon}\cdot a_t) d\mu_X = \frac{\textrm{Vol }Y}{\textrm{Vol }X}\left(\int_X w d\mu_X\right) + O(\varepsilon^{p'}\|w\|_l) + O(\|w\|_l \varepsilon^{-C_l} e^{-t\delta_0'})$$ 
for all $l\geq \max\{l_0',\lfloor\textrm{dim}(X)/2\rfloor+1\}$. 

By plugging \eqref{e.wavefront} into the estimate above, we conclude that 
$$\int_{Y a_t} w d\mu_{Y a_t} = \frac{\textrm{Vol }Y}{\textrm{Vol }X}\left(\int_X w d\mu_X\right) + O(\varepsilon^{p'}\|w\|_l) + O(\|w\|_l \varepsilon^{-C_l} e^{-t\delta_0'})$$ 
for all $l\geq l_0:=\max\{l_0',\lfloor\textrm{dim}(X)/2\rfloor+2\}$. 

By taking $\varepsilon:=e^{-\delta_0'' t}$ and by optimizing\footnote{I.e., we choose $\delta_0''>0$ so that $\varepsilon^{p'} = \varepsilon^{-C_{l_0}}e^{-t\delta_0'}$.} the value of $\delta_0''$, we obtain that 
$$\int_{Y a_t} w d\mu_{Y a_t} = \frac{\textrm{Vol }Y}{\textrm{Vol }X}\left(\int_X w d\mu_X\right) + O(\|w\|_{l_0} e^{-t\delta_0})$$ 
for $l_0:=\max\{l_0',\lfloor\textrm{dim}(X)/2\rfloor+2\}$ and $\delta_0:=\frac{p'}{p'+C_{l_0}}\delta_0'$. 

Since $0<p'<1$ is an arbitrary parameter, we deduce that \eqref{e.thm431} holds for $l_0:=\max\{l_0',\lfloor\textrm{dim}(X)/2\rfloor+2\}$ and any choice of 
\begin{equation}\label{e.delta0}
0<\delta_0<\frac{\delta_0'}{1+2l_0+4\textrm{ dim}(Y)+\frac{\textrm{dim}(X)}{2}}
\end{equation}

\section{Rates of mixing and representation theory}\label{s.representations} 

\begin{definition} 
1. A {\it unitary representation} $\pi$ of $G$ in a (separable) Hilbert space $\mathcal{H}_\pi$ is a morphism $G \to \mathrm{U} (\mathcal{H}_\pi )$ such that for any $v \in \mathcal{H}_\pi$ the map $G \to \mathcal{H}_\pi$; $g \mapsto \pi (g) v$ is continuous. If this map is smooth one says that $v$ is a $C^\infty$-vector of $\pi$. We denote by $\mathcal{H}_\pi^\infty$ the set of $C^\infty$-vectors of $\pi$.

2. Given two vectors $v,w \in \mathcal{H}_\pi$, we define the {\it matrix coefficient} $c_{v,w} : G \to \C$ of $\pi$ as the continuous map $g \mapsto \langle \pi (g) v , w \rangle$. The coefficient $c_{v,w}$ is said to be {\it $K$-finite} if both the vector spaces generated by $\pi (K) \cdot v$ and $\pi (K) \cdot w$ are finite dimensional.

3. Let $p(\pi)$ be the infimum of the set of real numbers $p \geq 2$ such that all $K$-finite matrix coefficients of $\pi$ are in $L^p (G)$. 

4. Say that a unitary representation $\sigma$ of $G$ is {\it weakly contained} in $\pi$ if any matrix coefficient of $\sigma$ can be obtained as the limit, with respect to the topology of uniform convergence on compact subsets, of a sequence of matrix coefficients of $\pi$.
\end{definition}

Given an element $g=nak \in G$, we write $a=e^{H(g)}$. The {\it Harish-Chandra} function is $\Xi = \Xi_G : G \to \mathbb{R}$ defined by 
$$\Xi (g) = \int_K e^{-\rho (H(kg^{-1}  ))} dk$$
where $\rho$ is half the sum of the positive restricted roots counting multiplicities.
The Harish-Chandra function decreases exponentially fast along $A^+$; modulo a polynomial factor of a logarithmic argument, it decreases like $e^{-\rho (H)}$. 

Let $d=\textrm{dim}(K)$ be the dimension of $K$ and fix a basis $\mathcal{B}$ of the Lie algebra $\mathfrak{k}$ of $K$. Given a smooth vector $v \in \mathcal{H}_\pi^\infty$ we set
$$S(v) = \sum_{\mathrm{ord} (D) \leq \lfloor d/2\rfloor +1} ||\pi (D) v ||,$$
where $D$ varies among all monomials in elements of $\mathcal{B}$ of degree $\leq \lfloor d/2\rfloor+1$ and, if $X_1 , \ldots , X_r$ are elements of $\mathcal{B}$, we have $\pi (X_1 \cdots X_r ) = \pi (X_1) \cdots \pi (X_r)$ and each $\pi (X_i )$ acts by derivation. 

\begin{proposition} \label{Prep}
For all positive $\varepsilon$ and $k \in \mathbb{N}^*$, there exists a constant $C= C(\varepsilon ,k )$ such that if $\pi$ is a unitary representation of $G$ with $p(\pi ) \leq 2k$, then for all $v , w \in \mathcal{H}_\pi^\infty $ and for all positive $t$ we have:
\begin{equation} 
|\langle \pi (a_t ) v , w \rangle | \leq C S(v) S(w) e^{-(p/k-\varepsilon) t},
\end{equation}
where $p= \rho (H)$ and $H$ is the infinitesimal generator of the one-parameter subgroup $(a_t)$. 
\end{proposition}
\begin{proof} Up to replacing $\pi$ by the tensor product $\pi^{\otimes k}$ we may suppose that $k=1$; see \cite[p. 108]{CHH}. It then follows from \cite[Theorem 1]{CHH} that $\pi$ is weakly contained in the (right) regular representation $L^2 (G)$. We are then reduced to prove the proposition in the case where $\pi$ is the regular representation of $G$ (and $k=1$); see the proof of \cite[Theorem 2]{CHH} for more details on this last reduction. 

Now consider $v$ and $w$ in $L^2 (G) \cap C^{\infty} (G)$. The functions 
$$\varphi : x \mapsto \sup_{k \in K} | v(xk) | \mbox{ and } \psi : x \mapsto \sup_{k \in K} |w (xk) | $$
are both positive and $K$-invariant, and we have:
$$| \langle \pi (a_t ) v , w\rangle_{L^2 (G)} | \leq \int_G \varphi (x a_t) \psi (x) dx = | \langle \pi (a_t ) \varphi , \psi \rangle_{L^2 (G)} |.$$
Now the Sobolev lemma (see \cite[Proposition 2.6]{Moore}) implies that the $L^\infty $ norms of $\varphi$ and $\psi$ can be estimated in terms of their Sobolev norms along $K$. More precisely: there exists a constant $c$ such that the for all $x \in G$, 
$$\varphi(x)^2 = \sup_{k \in K} | v(xk) |^2  \leq c \sum_{\mathrm{ord} (D) \leq \lfloor d/2\rfloor +1} ||(\rho (D) v) (x \cdot ) ||_{L^2 (K)}.$$
Integrating over $G$ (here we assume for simplicity that the measure of $K$ is $1$) one concludes that 
$||\varphi||_{L^2(G)} \leq \sqrt{c} S(v)$ and similarly for $\psi$. It remains to prove that there exists a constant $d_\varepsilon$ such that if $\varphi, \psi \in L^2 (G)$ are two $K$-invariant, positive functions of norm $1$, then 
$$|\langle \pi (a_t ) \varphi , \psi \rangle | \leq d_\varepsilon e^{- (p/k + \varepsilon ) t}.$$
First it follows from the computations of \cite[pp. 106-107]{CHH} that
\begin{equation*}
\begin{split}
|\langle \pi (g) \varphi , \psi \rangle | & = \int_K \left( \int_{NA} \varphi (na) \psi (nakg^{-1}) e^{2 \rho (\log a)} dn da \right) dk \\
& \leq ||\varphi ||_{L^2(G)} \int_K \left( \int_{NA} \psi (na H(kg^{-1}))^2 e^{2\rho (\log a)} dn da \right)^{1/2} dk \\
& =  ||\varphi ||_{L^2(G)} \cdot ||\psi ||_{L^2(G)} \int_K e^{-\rho (H(kg^{-1}  ))} dk = ||\varphi ||_{L^2(G)} \cdot ||\psi ||_{L^2(G)}\Xi (g).
\end{split}
\end{equation*}
Now recall that, up to ``polynomial factors of logarithmic arguments'', the function $\Xi (a_t)$ decreases like $e^{-t \rho (H)} = e^{-pt}$. The proposition follows.
\end{proof}

We shall apply this proposition to the (quasi-)regular representation $\pi$ of $G$ in the subspace $L^2_0 (\Gamma \backslash G)$ of $L^2 (\Gamma \backslash G)$ that is orthogonal to the space of constant functions. It follows from \cite{Li} that $p(\pi )= 20$.  Proposition \ref{Prep} therefore applies with $k=10$. Note that in our case $p=10$. 

Now let $\alpha$ and $\beta$ be two smooth functions in $L^2 (X)$ then 
$$\alpha_0 := \alpha - \int_X \alpha d\mu \mbox{ and } \beta_0 := \beta - \int_X \beta d\mu \in L^2_0 (X)$$
and we have:
$$\langle \pi (g) \alpha_0 , \beta_0 \rangle_{L_0^2 (X)} = \int_X \alpha\cdot(\beta\cdot g) d\mu - \left(\int_X \alpha d\mu\right)\left(\int_X \beta d\mu\right).$$
From Proposition \ref{Prep} and the fact that $S(\alpha ) \leq ||\alpha ||_{\lfloor d/2\rfloor+1}$ we conclude that 
$$\left| \int_X \alpha\cdot(\beta\cdot a_t ) d\mu - \left(\int_X \alpha d\mu\right)\left(\int_X \beta d\mu\right) \right| = O(\|\alpha\|_l \|\beta\|_l e^{-t \delta_0'})$$
for any $l \geq l_0 ' := \lfloor \textrm{dim}(K)/2\rfloor+1$ and any $\delta_0 ' < 1$.

\section{End of proof of Theorem \ref{t.A}} 

The explicit value of $\delta$ announced in Theorem \ref{t.A} can be easily derived from the discussion above. Indeed, we just saw in Section \ref{s.representations} that $\delta_0'=1-$ and $l_0'=\lfloor\textrm{dim}(K)/2\rfloor+1$. Because $174=\textrm{dim}(K)<\textrm{dim}(X)=231$ and $\textrm{dim}(Y)=210$, we deduce from \eqref{e.delta0} that $l_0=\lfloor\textrm{dim}(X)/2\rfloor+2=117$ and 
$$\delta_0 = \left( \frac{1}{1+2\times 117+4\times 210+\frac{231}{2}} \right)^- = \left( \frac{2}{2381} \right)^-$$
Finally, by inserting these informations into \eqref{e.dl} and \eqref{e.delta}, we conclude that  
$$\delta=\frac{\delta_0}{l_0+\frac{57}{2}+1} = \left( \frac{4}{697633} \right)^- \approx ( 
5.7336737224\dots\times 10^{-6})^-.$$

\end{document}